\documentclass[11pt,a4paper]{article}

\usepackage{amsmath}
\usepackage{amssymb,verbatim}
\usepackage{amsfonts}
\usepackage{amsthm}
\usepackage{pifont}
\usepackage{graphicx}

\newtheorem{theorem}{Theorem}
\newtheorem{lemma}{Lemma}

\newtheorem{corollary}{Corollary}

\newcommand{\beq}{\begin{equation*}}
\newcommand{\eeq}{\end{equation*}}
\newcommand{\beqn}{\begin{equation}}
\newcommand{\eeqn}{\end{equation}}
\newcommand{\Po}{\mathcal{P}_{n,\;\!r}}
\newcommand{\M}{\mathcal{S}(g,\:\!s)}
\newcommand{\RR}{\mathbb R}
\newcommand{\NN}{\mathbb N}
\newcommand{\dd}{\mathrm{d}}
\newcommand{\n}{\natural}

\begin{document}

\title{Random chords and point distances\\ in regular polygons}
\author{Uwe B\"asel}
\date{}
\maketitle

\begin{abstract}
\noindent In this paper we obtain the chord length distribution function for any regular polygon. From this function we conclude the density function and the distribution function of the distance between two uniformly and independently distributed random points in the regular polygon. The method to calculate the chord length distribution function is quite different from those of Harutyunyan and Ohanyan, uses only elementary methods and provides the result with only a few natural case distinctions.\\[0.2cm]
\end{abstract}

\section{Introduction}

A random line $g$ intersecting a convex set $\mathcal{K}$ in the plane produces a chord of $\mathcal{K}$. The length $s$ of this chord is a random variable. If the motion invariant line measure (see below) is used for the definition of the line, the expectation of the chord length is equal to $\pi A/u$ where $A$ is the area of $\mathcal{K}$ and $u$ the length of its perimeter \cite[p.\ 30]{Santalo}. The chord length distribution function of a regular triangle was calculated by Sulanke \cite[p.\ 57]{Sulanke}. Harutyunyan and Ohanyan \cite{HO} calculated the chord length distribution function for regular polygons using Dirac's $\delta$-function in Pleijel's identity. Bertrand's paradox associated with the chord length distribution of a circle is well known \cite[pp.~116-118]{Czuber}, \cite[pp.\ 172-179]{Mathai}.  

The distance $t$ between two points chosen independently and uniformly at random from $\mathcal{K}$ is also a random variable. Borel \cite{Borel} considered this distance in elementary geometric figures such as triangles, squares and so on (see \cite[p.\ 163]{MMP}). The expectations for the distance between two random points for an equilateral triangle and a rectangle are to be found in \cite[p.~49]{Santalo}. Ghosh \cite{Ghosh} derived the distance distribution for a rectangle. There are a lot of results concerning the distance $t$  within a convex set or in two convex sets, see Chapter 2 in \cite{Mathai}.

The moments of $s$ and $t$ resp. are closely connected by a simple formula \cite[pp.\ 46/47]{Santalo}. The second moments of the chord length for regular polygons have been obtained by Heinrich \cite{Heinrich}.

For practical applications of chord lenghts and point distances of convex sets in physics, material sciences, operations research and other fields see \cite{Gille} and \cite{Marsaglia}.

The first aim of the present paper is to derive the chord length distribution function for any regular polygon in a simple form with only a few natural case distinctions using a method that requires only elementary geometric considerations and elementary integrations (especially not using Dirac's $\delta$-function in Pleijels identity as done in \cite{HO}). Our method is also suitable for irregular and even (with slight modifications) non-convex polygons as shown in \cite{Baesel_Duma}. The second aim is to conclude the density function and the distribution function of the distance between two random points in every regular polygon. This result is new to the author's knowledge.  

\begin{figure}[h]
  \vspace{0.2cm}
  \begin{center}
    \includegraphics[scale=1]{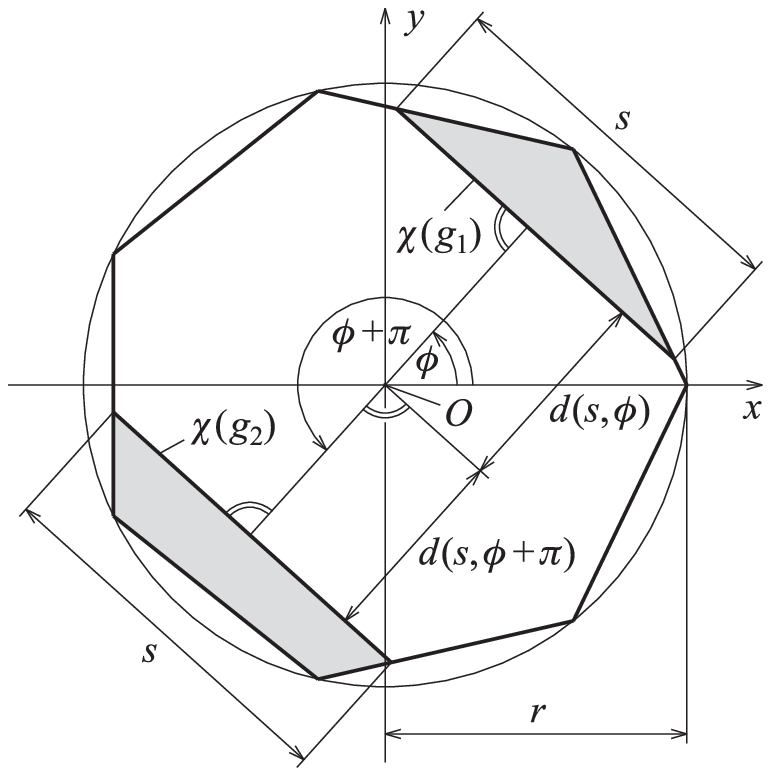}
  \end{center}
  \vspace{-0.2cm}
  \caption{\label{Fig_1} The polygon $\Po$ (example $n=7$)}
\end{figure}
We denote by $\Po$ the regular polygon with $n$ sides and circumscribed circle with radius $r$ and centre point in the origin $O$ (see Fig.\ \ref{Fig_1}). A straight line $g$ in the plane is determined by the angle $\phi$, $0\leq\phi<2\pi$, that the direction perpendicular to $g$ makes with a fixed direction (e.g. the $x$-axis) and by its distance $p$, $0\leq p<\infty$, from the origin $O$:
\[ g = g(p,\phi) = \{(x,y)\in\RR^2:\, x\cos\phi+y\sin\phi = p\}\,. \]
The measure $\mu$ of a set of lines $g(p,\phi)$ is defined by the integral, over the set, of the differential form $\dd g=\dd p\:\dd\phi$. Up to a constant factor, this measure is the only one that is invariant under motions in the Euclidean plan \cite[p.~28]{Santalo}.

The chord length distribution function of $\Po$ is usuallay defined as
\beq
  F(s) = \frac{1}{u}\:\mu(\{g:\,g\cap\Po\not=\emptyset,\,|\chi(g)|\leq s\})\,,
\eeq
where $\chi(g)=g\cap\Po$ is the chord of $\Po$, produced by the line $g$, $|\chi(g)|$ the length of $\chi(g)$, and $u$ the length of the perimeter of $\Po$. (The measure of all lines $g$ that intersect a convex set is equal to its perimeter \cite[p. 30]{Santalo}.) We use the distribution function in the form
\beqn \label{F}
  F(s) = 1-\frac{1}{u}\:\mu(\{g:\,g\cap\Po\not=\emptyset,\,|\chi(g)|>s\})
\eeqn
(cf.\ \cite[p.\ 161]{AMS}). So it remains to calculate the measure of all lines that produce a chord of length $|\chi(g)|>s$\,. Using the abbreviation
\beq
  \M:=\{g:\,g\cap\Po\not=\emptyset,\,|\chi(g)|>s\}\,,
\eeq
we have
\beq
 \mu(\M) = \int_{\M}\,\dd g = \int_{\M}\,\dd p\:\dd\phi\,.
\eeq
We consider all lines $g$, having a direction perpendicular to a fixed angle $\phi\in[0,\pi)$ with $g\cap\Po\not=\emptyset$. Among these lines there are in almost all cases two lines $g_1$ and $g_2$ with chords of equal length $s$ (see Fig.\ \ref{Fig_1}). All parallel lines $g$ lying in the strip between $g_1$ and $g_2$ have a chord with length $|\chi(g)|>s$. The breadth of this strip is equal to $d(s,\phi)+d(s,\phi+\pi)$, where $d(s,\phi)$ and $d(s,\phi+\pi)$ are the distances between $O$ and $g_1$ and $O$ and $g_2$ repectively. So we have
\beqn \label{Mass}
  \mu(\M) = \int_0^{\pi}\,\big[d(s,\phi)+d(s,\phi+\pi)\big]\,\dd\phi\,.
\eeqn

\section{The distance function}

In the following we determine the distance function in formula \eqref{Mass}
\beqn \label{d}
  d: \; [0,\max(s)]\times[0,\infty)\;\rightarrow\;[0,r] \,,\;\;\; (s,\phi)\;\mapsto\; d(s,\phi)\,,
\eeqn
where $\max(s)$ is the maximum chord length $s$ in $\Po$. If no chord of length $s$ in the direction perpendicular to $\phi$ exists, we put $d(s,\phi)=0$. Of course for fixed value of $s$, $d(s,\,\cdot\,)$ is a $2\pi/n$-periodic function. 

We put
\beq
  K = \bigg\lfloor\frac{n-2}{2}\bigg\rfloor\,,
\eeq
where $\lfloor\cdot\rfloor$ is the integer part of $\,\cdot\,$, and define the function
$m\!:\,\NN\times\NN\rightarrow\NN$ by
\beq
  m(k,n) = \left\{
  \begin{array}{lcl}
	k\;\mbox{mod}\;n & \mbox{if} & k\;\mbox{mod}\;n\not=0\,,\\[0.05cm]
	n & \mbox{if} & k\;\mbox{mod}\;n = 0\,.
  \end{array}\right. 
\eeq
\begin{figure}[h]
  \begin{center}
    \includegraphics[scale=0.9]{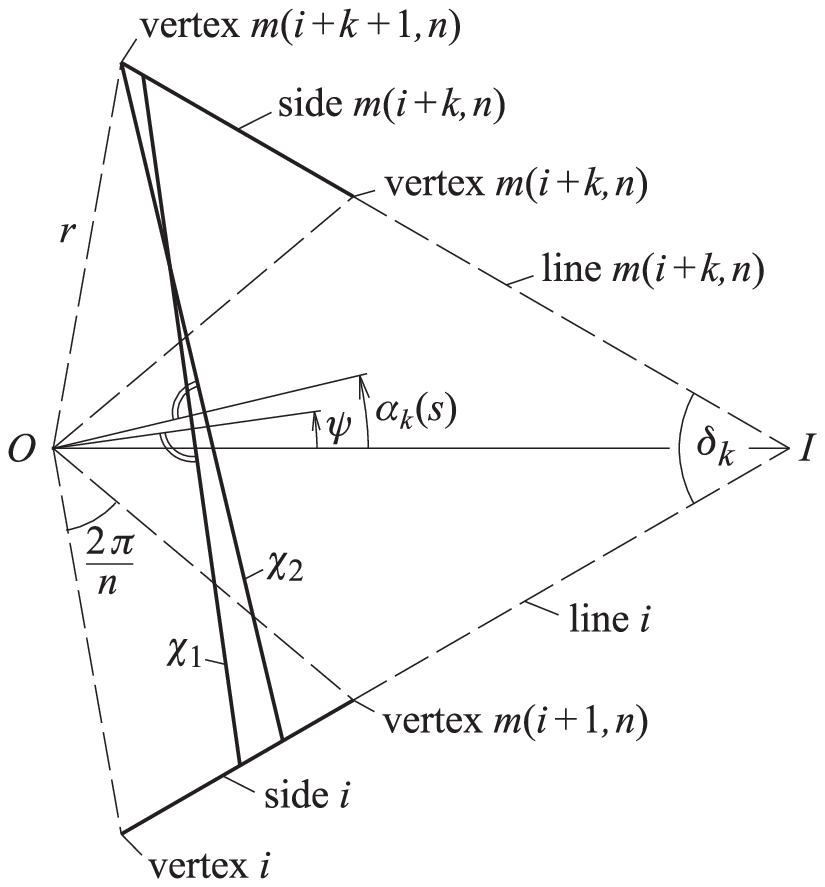}
  \end{center}
  \vspace{-0.3cm}
  \caption{\label{Fig_2} Chords $\chi$ between side $i$ and side $m(i+k,n)$}
\end{figure}

\noindent
The angle $\delta_k$ $\mbox{(see Fig. \ref{Fig_2})}$ between the lines $i$ and $m(i+k,n)$ containing the sides $i$, $i=1,\ldots,n$, and $m(i+k,n)$ of $\Po$ is given by
\beq
  \delta_k = \bigg(1-\frac{2k}{n}\bigg)\,\pi \;,\quad k=1,\ldots,K^*\,,
\eeq
where
\beq
  K^* = \left\{
  \begin{array}{ll}
	K+1 & \mbox{if $n$ is odd}\,,\\[0.05cm]
	K & \mbox{if $n$ is even}\,.  
  \end{array}\right.
\eeq
The distance $\ell_k$ between the vertices $i$ and $m(i+k,\,n)$ is for $k=0,\ldots,K+1$ given by  
\beq
  \ell_k = 2r\sin\frac{k\pi}{n}\,.
\eeq
The maximum chord length in $\Po$ is equal to $\ell_{K+1}$. For the distance $x$ between one point of side $i$ and one point of side $m(i+k,n)$, $k=1,\ldots,K^*$, we have $\ell_{k-1}\leq x\leq\ell_{k+1}$, and $\ell_k\leq x\leq\ell_{k+2}$ for the analogous distance of the sides $i$ and $m(i+k+1,n)$. Therefore, a chord of length $s$, $\ell_k\leq s\leq\ell_{k+1}$, is a chord between two sides $i$ and $m(i+k,n)$ or two sides $i$ and $m(i+k+1,n)$. 
 
In the first step we derive formulas for the distance $d_k^*(s,\psi)$ between $O$ and a chord $\chi$ of length $s$, $\ell_k\leq s\leq\ell_{k+1}$, $k=0,\ldots,K$, where $\psi$ denotes the oriented angle between the segment from $O$ to the intersection point $I$ of the lines $i$ and $m(i+k,n)$ and the line perpendicular to $\chi$ ($\mbox{Fig. \ref{Fig_2})}$. We only consider the interval $0\leq\psi\leq\pi/n$. It is necessary to distinguish the following cases:\\[0.25cm]
\underline{Case 1}
\beqn \label{case1}
  \ell_k\leq s<\ell_{k+1} \;\mbox{{\em with}}\:\left\{\!
  \begin{array}{l@{\;\in\;}l}
	k & \{1,\ldots,K-1\} \:\,\mbox{{\em if $n$ is even}}\,,\\[0.1cm]
	k & \{1,\ldots,K\} \:\,\mbox{{\em if $n$ is odd and $s\leq 2r\cos^2\dfrac{\pi}{2n}$}}
  \end{array}\!\right\}
\eeqn
For $0\leq\psi\leq\alpha_k(s)$ (e.g. for position $\chi_1$ of $\chi$ in $\mbox{Fig. \ref{Fig_2})}$ the distance $d_k^*(s,\psi)$ between $O$ and $\chi$ is equal to 
\beqn \label{q}
  q_k(s,\psi) := r\cos\dfrac{\pi}{n}\sec\dfrac{k\pi}{n}\cos\psi
	-\dfrac{s}{2}\,\bigg(\tan\dfrac{k\pi}{n}\cos^2\psi-\cot\dfrac{k\pi}{n}\sin^2\psi\bigg)
\eeqn
 The angle $\alpha_k$ is determined by the position $\chi_2$ of $\chi$ with the upper end-point in the vertex $m(i+k+1,n)$:
\beqn \label{alpha}
  \alpha_k(s) = \arcsin\bigg(\dfrac{2r}{s}\sin\dfrac{k\pi}{n}\sin\dfrac{(k+1)\pi}{n}\bigg)-\dfrac{k\pi}{n}\,.
\eeqn
For $\alpha_k(s)\leq\psi\leq\pi/n$, $\chi$ is a chord between the sides $i$ and $m(i+k+1,n)$. So we find
\beqn \label{c1}
  d_k^*(s,\psi) = \left\{
  \begin{array}{lcl}
	q_k(s,\psi) & \mbox{if} & 0\leq\psi\leq\alpha_k(s)\;,\\[0.15cm]
	q_{k+1}(s,\psi-\pi/n) & \mbox{if} & \alpha_k(s)<\psi\leq\pi/n\;.
  \end{array}\right\}
\eeqn
\noindent
\underline{Case 1a}
\beqn \label{case1a}
  0=\ell_0\leq s<\ell_1 \quad\mbox{{\em and}}\quad s<2r\cos^2\frac{\pi}{2n}
\eeqn
We have $\alpha_0(s)=0$ if $s\not=0$, and the limit of $\alpha_0(s)$ at $s=0$ is 0. Therefore we get
\beqn \label{c1a}
  d_0^*(s,\psi)=q_1(s,\psi-\pi/n) \quad\mbox{for}\quad 0\leq\psi\leq\pi/n
\eeqn 
as special case of case 1 with $\alpha_0(s)=0$ in formula \eqref{c1}.\\[0.25cm]
\underline{Case 2}
\beqn \label{case2}
  \mbox{{\em $n$ is even and $\ell_K\leq s\leq\ell_{K+1}$}}
\eeqn  
A chord $\chi$ in the direction perpendicular to $\psi$ does not exist if $\alpha_K(s)<\psi\leq\pi/n$, therefore  
\beqn \label{c2} 
  d_K^*(s,\psi) = \left\{
  \begin{array}{lcl}
	q_K(s,\psi) & \mbox{if} & 0\leq\psi\leq\alpha_K(s)\;,\\[0.15cm]
	0 & \mbox{if} & \alpha_K(s)<\psi\leq\pi/n\;. 
  \end{array}\right\}
\eeqn
\underline{Case 3}
\beqn \label{case3}
  \mbox{{\em $n$ is odd and $\ell_K\leq s\leq\ell_{K+1}$ and $s\geq 2r\cos^2\dfrac{\pi}{2n}$}}
\eeqn
A chord $\chi$ in the direction perpendicular to $\psi$ does not exist if
\beq
  \beta(s)<\psi<\frac{\pi}{n}-\beta(s)\,,
\eeq
with
\beqn \label{beta}
  \beta(s) = \dfrac{\pi}{2n}-\arccos\bigg(\dfrac{2r}{s}\cos^2\dfrac{\pi}{2n}\bigg)\,; 
\eeqn
therefore
\beqn \label{c3}
  d_K^*(s,\psi) = \left\{
  \begin{array}{lcl}
	q_K(s,\psi) & \mbox{if} & 0\leq\psi<\alpha_K(s)\;,\\[0.15cm]
	q_{K+1}(s,\psi-\pi/n) & \mbox{if} & \alpha_K(s)\leq\psi\leq\beta(s)\;,\\[0.15cm]
	0 & \mbox{if} &\beta(s)<\psi<\pi/n-\beta(s)\;,\\[0.15cm]
	q_{K+1}(s,\psi-\pi/n) & \mbox{if} & \pi/n-\beta(s)\leq\psi\leq\pi/n\;. 
  \end{array}\right\}
\eeqn
Due to the symmetry of the graph of $d_k^*(s,\psi)$ with respect to the line $\psi=\pi/n$, the values in the interval $\pi/n<\psi\leq 2\pi/n$ can be easily calculated from \eqref{c1}, \eqref{c1a}, \eqref{c2} and \eqref{c3} with
\beqn \label{symmetry}
  d_k^*(s,\psi) = d_k^*(s,2\pi/n-\psi)\,.
\eeqn
Since $d_k^*(s,\psi)$ is a $2\pi/n$-periodic function, we get the values for $2\pi/n<\psi<\infty$ with the translation
\beq
  \psi\;\mapsto\;\psi-\delta(\psi) \quad\mbox{with}\quad
  \delta(\psi)=\bigg\lfloor\frac{n\psi}{2\pi}\bigg\rfloor\frac{2\pi}{n}\,.
\eeq
 
In the case of even $n$, the substitution $\psi=\phi+\pi/n$ yields the distances for angle $\phi$ starting from a vertex as shown in Fig.\ \ref{Fig_1}. In the case of odd $n$ we have $\psi=\phi$. So we have found:  
\begin{lemma}
The restrictions $d_k(s,\phi)=d(s,\phi)|_{\ell_k\,\leq\,s\,<\,\ell_{k+1}}$ of the distance function $d$ are given by
\beq
  d_k(s,\phi) = \left\{
  \begin{array}{ll}
	d_k^*\big(s,\phi-\delta(\phi)\big) & \mbox{if $n$ is odd}\,,\\[0.2cm]
	d_k^*\bigg(s,\phi+\dfrac{\pi}{n}-\delta\Big(\phi+\dfrac{\pi}{n}\Big)\bigg) & \mbox{if $n$ is even}\,,
  \end{array}\right.
\eeq
for $k=0,\ldots,K$, where
\beq
  \delta(\,\cdot\,)=\bigg\lfloor\frac{n\:\cdot\,}{2\pi}\bigg\rfloor\frac{2\pi}{n}
\eeq
and
$d_k^*$ according to the formulas $\eqref{c1}$, $\eqref{c2}$, $\eqref{c3}$ and $\eqref{symmetry}$
with $\alpha_k$ and $\beta$ according to \eqref{alpha} and \eqref{beta} respectively.
\end{lemma}

\section{Chord length distribution function}

So we can write the chord length distribution function \eqref{F} in the form
\beq
  F(s) = \left\{
  \begin{array}{lcl}
	0 & \mbox{if} & -\infty<s<\ell_0=0\,,\\[0.15cm]
	H_k(s) & \mbox{if} & \ell_k\leq s<\ell_{k+1} \;\;\mbox{for}\;\; k=0,\ldots,K,\\[0.15cm]
	1 & \mbox{if} & \ell_{K+1}\leq s<\infty\,,
  \end{array}\right. 
\eeq
where
\[ H_k(s) = 1-\frac{\mu_k(s)}{2nr\sin(\pi/n)} \;\;\mbox{with}\;\;
	\mu_k(s):=\int_0^{\pi}\big[d_k(s,\phi)+d_k(s,\phi+\pi)\big]\,\dd\phi\,. \]
With $(\star)$ the $2\pi/n$-periodicity of $d_k(s,\phi)$ and  $(\diamond)$ the symmetry of $d_k(s,\phi)$ with respect to the line $\phi=\pi/n$,
we find for odd and even $n$  
\begin{align*}
  \mu_k(s)
	= {} & \int_0^{\pi}\big[d_k(s,\phi)+d_k(s,\phi+\pi)\big]\,\dd\phi\\
	\overset{\star}{=} {} & \frac{n}{2}\int_0^{2\pi/n}\big[d_k(s,\phi)+d_k(s,\phi+\pi)\big]\,\dd\phi\\
	= {} & \frac{n}{2}\int_0^{2\pi/n}\bigg[d_k(s,\phi)+d_k\bigg(s,\phi+n\,\frac{2\pi}{n}\bigg)\bigg]\,\dd\phi\\
	\overset{\star}{=} {} & n\int_0^{2\pi/n}d_k(s,\phi)\,\dd\phi
	\overset{\diamond}{=} 2n\int_0^{\pi/n}d_k(s,\phi)\,\dd\phi\,.
\end{align*}
Note that this integral formula (together with the piecewise definition of the distance function) allows us to calculate the distribution function of every regular polygon in a rather simple way.
  
The indefinite integral of the function $q_k$ (see \eqref{q}) is given by
\begin{align} \label{J_k}
  J_k(s,\phi)
	= {} & \int q_k(s,\phi)\,\dd\phi = r\cos\dfrac{\pi}{n}\sec\dfrac{k\pi}{n}\sin\phi\nonumber\\
	  {} & -\dfrac{s}{8}\bigg[(2\phi+\sin 2\phi)\tan\dfrac{k\pi}{n}
	  		-(2\phi-\sin 2\phi)\cot\dfrac{k\pi}{n}\bigg].
\end{align}
In case 1 (see \eqref{case1}) we get with $\alpha_k=\alpha_k(s)$
\begin{align*}
  \frac{\mu_k(s)}{2n}
	= {} & \int_0^{\alpha_k}d_k(s,\phi)\,\dd\phi+\int_{\alpha_k}^{\pi/n}d_k(s,\phi)\,\dd\phi\\
	= {} & \int_0^{\alpha_k}q_k(s,\phi)\,\dd\phi+\int_{\alpha_k}^{\pi/n}q_{k+1}\bigg(s,\phi-\frac{\pi}{n}\bigg)\,\dd\phi
\end{align*}
and with the substitution $\phi^*=\phi-\pi/n$
\begin{align*} \label{F_k}
  \frac{\mu_k(s)}{2n}
	= {} & \int_0^{\alpha_k}q_k(s,\phi)\,\dd\phi+\int_{\alpha_k-\pi/n}^{0}q_{k+1}(s,\phi^*)\,\dd\phi^*\displaybreak[0]\\
	= {} & J_k(s,\alpha_k)-J_k(s,0)+J_{k+1}(s,0)-J_{k+1}\bigg(s,\alpha_k-\frac{\pi}{n}\bigg)\\
	= {} & J_k(s,\alpha_k)-J_{k+1}\bigg(s,\alpha_k-\frac{\pi}{n}\bigg)\,,
\end{align*}
since $J_k(s,0)=0$. In case 1a (see \eqref{case1a}) one finds with $\alpha_0(s)=0$ and $J_k(s,-\phi)=-J_k(s,\phi)$
\beq
  \frac{\mu_0(s)}{2n} = -J_1\bigg(s,-\frac{\pi}{n}\bigg) = J_1\bigg(s,\frac{\pi}{n}\bigg)\,.
\eeq
Putting $J_0(s,0)=0$, this formula can be considered as special case of the formula for case 1. In case~2 (see \eqref{case2}) one easily finds
\beq
  \frac{\mu_K(s)}{2n} = J_K(s,\alpha_K)
\eeq
and in case 3 (see \eqref{case3}) with $\beta=\beta(s)$
\begin{align*}
   \frac{\mu_K(s)}{2n}
 	= {} & \bigg[\int_0^{\alpha_K}+\int_{\alpha_K}^{\beta}+\int_{\beta}^{\pi/n-\beta}
 			+\int_{\pi/n-\beta}^{\pi/n}\bigg]\;d_K(s,\phi)\,\dd\phi\\
 	= {} & \bigg[\int_0^{\alpha_K}+\int_{\alpha_K}^{\beta}
 			+\int_{\pi/n-\beta}^{\pi/n}\bigg]\;d_K(s,\phi)\,\dd\phi\\
 	= {} & \int_0^{\alpha_K}q_K(s,\phi)\,\dd\phi+\bigg[\int_{\alpha_K}^{\beta}
			+\int_{\pi/n-\beta}^{\pi/n}\bigg]\;q_{K+1}(s,\phi-\pi/n)\,\dd\phi\\
 	= {} & \int_0^{\alpha_K}q_K(s,\phi)\,\dd\phi+\bigg[\int_{\alpha_K-\pi/n}^{\beta-\pi/n}
			+\int_{-\beta}^{0}\bigg]\;q_{K+1}(s,\phi^*)\,\dd\phi^*\\
 	= {} & J_K(s,\alpha_K)-J_{K+1}\bigg(s,\alpha_K-\frac{\pi}{n}\bigg)
 			+J_{K+1}\bigg(s,\beta-\frac{\pi}{n}\bigg)+J_{K+1}(s,\beta)\,.
\end{align*}
The function $J_k$ (see \eqref{J_k}) can be written as
\begin{align*}
  J_k(s,\phi)
	= {} & r\cos\dfrac{\pi}{n}\sec\dfrac{k\pi}{n}\sin\phi+\dfrac{s}{4}\left(2\phi\cot\dfrac{2k\pi}{n}
	  		-\sin(2\phi)\csc\dfrac{2k\pi}{n}\right).
\end{align*}
Furthermore, we write both the functions $\alpha_k$ (see \eqref{alpha}) and $\beta$ (see \eqref{beta}) in the form
\begin{align*}
  \arcsin\frac{a}{s}-b
\end{align*}
with
\beqn \label{A1B1} 
  a = A_1(k) = 2r\sin\dfrac{k\pi}{n}\sin\dfrac{(k+1)\pi}{n} \;,\quad
  b = B_1(k) = \dfrac{k\pi}{n}
\eeqn
for $\alpha_k$, and
\beqn \label{A2B2}
  a = A_2 = 2r\cos^2\dfrac{\pi}{2n}  \;,\quad
  b = B_2 = \dfrac{\pi}{2}\left(1-\dfrac{1}{n}\right)   
\eeqn
for $\beta$. Using some easy algebraic manipulations, one finds
\begin{align*}
  \frac{1}{r}\:J_k\left(s,\,\arcsin\frac{a}{s}-b\right)
    = {} & \Theta_1(k,a,b)\,s + \Theta_2(k,a,b)\,\frac{1}{s} + \Theta_3(k,a,b)\,\frac{\sqrt{s^2-a^2}}{s}\\
         & + \Theta_4(k,a,b)\,s\arcsin\frac{a}{s} =: h_k(s,a,b)\,, 
\end{align*}
where
\beqn \label{Theta}
 \left.\begin{array}{l@{\;=\;}l}
  \Theta_1(k,a,b) & \dfrac{1}{4r}\csc\dfrac{\pi}{n}\left(\sin(2b)\csc\dfrac{2k\pi}{n}
							-2b\cot\dfrac{2k\pi}{n}\right),\\[0.4cm]
  \Theta_2(k,a,b) & a\left(\cos b\cot\dfrac{\pi}{n}\sec\dfrac{k\pi}{n}
							-\dfrac{a}{2r}\sin(2b)\csc\dfrac{\pi}{n}\csc\dfrac{2k\pi}{n}\right),\\[0.4cm]
  \Theta_3(k,a,b) & -\!\left(\sin b\cot\dfrac{\pi}{n}\sec\dfrac{k\pi}{n}
							+\dfrac{a}{2r}\cos(2b)\csc\dfrac{\pi}{n}\csc\dfrac{2k\pi}{n}\right),\\[0.4cm]
  \Theta_4(k,a,b) & \dfrac{1}{2r}\csc\dfrac{\pi}{n}\cot\dfrac{2k\pi}{n}\,.							
 \end{array}\right\}
\eeqn
In summary, we have proved:
\begin{theorem} \label{Theorem}
The chord length distribution function $F$ of the regular polygon $\Po$ is given by
\beq
  F(s) = \left\{
  \begin{array}{lcl}
	0 & \mbox{if} & -\infty<s<\ell_0=0\,,\\[0.15cm]
	H_k(s) & \mbox{if} & \ell_k\leq s<\ell_{k+1} \;\;\mbox{for}\;\; k=0,\ldots,K,\\[0.15cm]
	1 & \mbox{if} & \ell_{K+1}\leq s<\infty\,,
  \end{array}\right. 
\eeq
where
\beq
  \ell_k = 2r\sin\frac{k\pi}{n} \;,\quad K = \bigg\lfloor\frac{n-2}{2}\bigg\rfloor
\eeq
and
\beq
  H_k(s) = \left\{
  \begin{array}{l}
	1-h_k(s,A_1(k),B_1(k))+h_{k+1}(s,A_1(k),B_1(k)+\pi/n)\\[0.15cm]
	\quad\mbox{if \, $(n$ is even $\,\wedge\,$ $k\in\{0,\ldots,K-1\})$ $\,\vee\,$
	$(n$ is odd $\,\wedge\,$ $s<\lambda)$}\,,\\[0.4cm]
	1-h_K(s,A_1(K),B_1(K))+h_{K+1}(s,A_1(K),B_1(K)+\pi/n)\\[0.15cm]
	-\:h_{K+1}(s,A_2,B_2+\pi/n)-h_{K+1}(s,A_2,B_2)\\[0.15cm]
	\quad\mbox{if \, $n$ is odd $\,\wedge\,$ $s\geq\lambda$}\,,\\[0.4cm]	
	1-h_K(s,A_1(K),B_1(K))\mbox{ \; if \; $n$ is even $\,\wedge\,$ $k=K$}
  \end{array}\right.
\eeq
with
\beq
  \lambda = 2r\cos^2\dfrac{\pi}{2n}\,,
\eeq
$A_1(k)$ and $B_1(k)$ according to \eqref{A1B1}, $A_2$ and $B_2$ according to \eqref{A2B2}, and 
\beq
  h_k(s,a,b) = \left\{\begin{array}{lcl}
	0 & \mbox{if} & k=0\,,\\[0.1cm]
	\sum_{i=1}^4\Theta_i(k,a,b)\,L_i(s,a) & \mbox{if} & k=1,2,\ldots,
  \end{array}\right.
\eeq
with $\Theta_i(k,a,b)$ according to \eqref{Theta},
and
\begin{align*}
  L_1(s,a) = {} & s \,,\;\; L_2(s,a) = \frac{1}{s} \,,\;\; L_3(s,a) = \frac{\sqrt{s^2-a^2}}{s} \,,\;\;
  L_4(s,a) = s\arcsin\frac{a}{s}\,.
\end{align*}
\end{theorem}
\vspace{0.2cm}
$F$ can be written in the form
\beq
  F(s) = H_0(s) = \bigg[\bigg(1-\frac{\pi}{n}\cot\frac{\pi}{n}\bigg)\csc\frac{\pi}{n}+\frac{\pi}{n}\sec\frac{\pi}{n}\bigg]\frac{s}{4r}
\eeq
for $0\leq s\leq\lambda$ if $n=3$, and $0\leq s\leq\ell_1$ if $n=$ 4, 5, \ldots. Note that this is a linear equation of $s$ (cf.\ \cite[pp. 866/867]{Gates}). 
 
From \cite[p.\ 55, Satz\ 2]{Sulanke} it follows that the chord length distribution function of  a regular polygon is a continuous function.

\section{Point distances}

In the following, we consider the distance between two uniformly and independently distributed random points within the polygon $\Po$ with perimeter~$u$ and area $A$:
\beq
  u = 2nr\sin\frac{\pi}{n} \,,\;\; A = \frac{1}{2}\,nr^2\sin\frac{2\pi}{n}\,.
\eeq
\begin{theorem} \label{Theorem2}
The density function $g$ of the distance $t$ between two random points in $\Po$ is given by
\beq
  g(t) = \left\{\begin{array}{lcl}
	\dfrac{2t}{A}\bigg[\pi+\dfrac{u}{A}\,\big(\phi^*(t)-t\big)\bigg] & \mbox{if} & t\in[0,\ell_{K+1})\,,\\[0.4cm]
	0 & \mbox{if} & t\in\mathbb{R}\setminus[0,\ell_{K+1})\,,
  \end{array}\right.
\eeq
where
\beq
  \phi^*(t) = \sum_{\nu=0}^{k-1}J_\nu^*(\ell_\nu,\ell_{\nu+1})+J_k^*(\ell_k,t)
	\quad\mbox{if}\quad \ell_k\leq t<\ell_{k+1},\; k=0,\ldots,K 
\eeq
with
\beq
  J_k^*(s,t) = H_k^*(t)-H_k^*(s)\,,
\eeq
where
\beq
  H_k^*(t) = \left\{
  \begin{array}{l}
	t-h_k^*(t,A_1(k),B_1(k))+h_{k+1}^*(t,A_1(k),B_1(k)+\pi/n)\\[0.15cm]
	\quad\mbox{if \, $(n$ is even $\,\wedge\,$ $k\in\{0,\ldots,K-1\})$ $\,\vee\,$
	$(n$ is odd $\,\wedge\,$ $t<\lambda)$}\,,\\[0.4cm]
	t-h_K^*(t,A_1(K),B_1(K))+h_{K+1}^*(t,A_1(K),B_1(K)+\pi/n)\\[0.13cm]
	-\:h_{K+1}^*(t,A_2,B_2+\pi/n)-h_{K+1}^*(t,A_2,B_2)\\[0.15cm]
	+\:h_{K+1}^*(\lambda,A_2,B_2+\pi/n)+h_{K+1}^*(\lambda,A_2,B_2)\\[0.15cm]
	\quad\mbox{if \, $n$ is odd $\,\wedge\,$ $t\geq\lambda$}\,,\\[0.4cm]	
	t-h_K^*(t,A_1(K),B_1(K))\mbox{ \; if \; $n$ is even $\,\wedge\,$ $k=K$}
  \end{array}\right.
\eeq
with
\beq
  h_k^*(t,a,b) = \left\{\begin{array}{l}
	0 \quad\mbox{if}\quad k=0 \,\vee\, (k\not=0\wedge t=0)\,,\\[0.1cm]
	\displaystyle{\sum_{i=1}^4\Theta_i(k,a,b)\,L_i^*(t,a)} \quad\mbox{if}\quad k\not=0\,\wedge\,t>0\,,
  \end{array}\right.
\eeq
\vspace{-0.4cm}
\begin{align*}
  L_1^*(t,a) = {} & \frac{t^2}{2} \;,\;\; L_2^*(t,a) = \ln t \;,\;\; 
  L_3^*(t,a) = \sqrt{t^2-a^2}+a\arcsin\frac{a}{t} \;,\\[0.1cm]
  L_4^*(t,a) = {} & \frac{1}{2}\left(a\,\sqrt{t^2-a^2}+t^2\arcsin\frac{a}{t}\right).
\end{align*}
\end{theorem}

\begin{proof}
According to \cite[p. 130]{Piefke}, the density function of the distance is given by
\beq
  g(t) = \frac{2ut}{A^2}\,\int_t^{\ell_{K+1}}(s-t)f(s)\,\dd s\,,
\eeq
where $f$ is the density function of the chord length. From integral geometry it is well-known that
\beq
  \int_0^{\ell_{K+1}}sf(s)\,\dd s = \frac{\pi A}{u}
\eeq
(see \cite[p.\ 47]{Santalo}, \cite[p.\ 94]{Mathai}), hence
\beq
  \int_t^{\ell_{K+1}}sf(s)\,\dd s = \frac{\pi A}{u}-\int_0^t sf(s)\,\dd s\,.
\eeq
Using integration by parts, we have
\beq
  \int_0^t sf(s)\,\dd s = sF(s)\Big|_0^t-\int_0^t F(s)\,\dd s = tF(t)-\int_0^t F(s)\,\dd s\,.
\eeq
Therefore, we obtain
\begin{align*}
  g(t)
	= {} & \frac{2t}{A}\left[\pi-\frac{u}{A}\left(t-\int_0^t F(s)\,\dd s\right)\right]
	= \frac{2t}{A}\left[\pi+\frac{u}{A}\left(\phi^*(t)-t\right)\right]
\end{align*}
with
\beq
  \phi^*(t) := \int_0^t F(s)\,\dd s\,.
\eeq
This yields for $\ell_k\leq t<\ell_{k+1}$, $k=0,\ldots,K$, 
\beq
  \phi^*(t) = \sum_{\nu=0}^{k-1}\,\int_{\ell_\nu}^{\ell_{\nu+1}}H_\nu(s)\,\dd s
				+ \int_{\ell_k}^t H_k(s)\,\dd s\\
\eeq
(in the case $k=0$, the sum is empty). With
\beq
  H_k^*(t) := \int H_k(t)\,\dd t \quad\mbox{und}\quad J_k^*(s,t) = H_k^*(t)-H_k^*(s) 
\eeq
it follows that
\begin{align*}
  \phi^*(t)
	= {} & \sum_{\nu=0}^{k-1}\left[H_\nu^*(\ell_{\nu+1})-H_\nu^*(\ell_\nu)\right]
			+ H_k^*(t)-H_k^*(\ell_k)\\
	= {} & \sum_{\nu=0}^{k-1}J_\nu^*(\ell_\nu,\ell_{\nu+1}) + J_k^*(\ell_k,t)\,.
\end{align*}
Furthermore, if $k\not=0$ and $t>0$,
\begin{align*}
  h_k^*(t,a,b)
	:= {} & \int h_k(t,a,b)\,\dd t = \sum_{i=1}^4\Theta_i(k,a,b)\int L_i(t,a)\,\dd t\\
	 = {} & \sum_{i=1}^4\Theta_i(k,a,b)\,L_i^*(t,a)
\end{align*}
with the indefinite integrals
\begin{align*}
  L_1^*(t,a) = {} & \int t\:\dd t = \frac{t^2}{2} \;,\quad L_2^*(t,a)=\int\frac{1}{t}\:\dd t = \ln t\,,\\
  L_3^*(t,a) = {} & \int\frac{\sqrt{t^2-a^2}}{t}\:\dd t = \sqrt{t^2-a^2}+a\arcsin\frac{a}{t}
\end{align*}
(see \cite[p.\ 48, Eq.\ 217]{Bronstein}) and, using integration by parts,
\begin{align*}
  L_4^*(t,a) = {} & \int L_4(t,a)\,\dd t = \int t\arcsin\frac{a}{t}\;\dd t\\
			 = {} & \frac{1}{2}\left(t^2\arcsin\frac{a}{t}+a\int\frac{t}{\sqrt{t^2-a^2}}\;\dd t\right)\\
			 = {} & \frac{1}{2}\left(t^2\arcsin\frac{a}{t}+a\,\sqrt{t^2-a^2}\right)\,.
\end{align*}

For odd $n$, the function
\beq
  H_K^*(t) = \left\{
  \begin{array}{lcl}
	H_{K,\,1}^*(t) & \mbox{if} & t<\lambda\,,\\[0.2cm]
	H_{K,\,2}^*(t) & \mbox{if} & t\geq\lambda
  \end{array}\right.
\eeq
with
\begin{align*}
  H_{K,\,1}^*(t) := {} & t-h_K^*(t,A_1(k),B_1(k))+h_{K+1}^*(t,A_1(K),B_1(K)+\pi/n)\,,\displaybreak[0]\\[0.1cm]
  H_{K,\,2}^*(t) := {} & t-h_K^*(t,A_1(K),B_1(K))+h_{K+1}^*(t,A_1(K),B_1(K)+\pi/n)\\
					   & -h_{K+1}^*(t,A_2,B_2+\pi/n)-h_{K+1}^*(t,A_2,B_2)
\end{align*}
is not continuous in $t=\lambda$. This causes a false result when calculating the integral
\beq
  J_K^*(\ell_K,\lambda) = H_K^*(\lambda) - H_K^*(\ell_K)\,.
\eeq
In order to avoid this problem (and unnecessary case distinctions), we define
\begin{align*}
  \widetilde{H}_{K,\,2}^*(t)
	= {} & H_{K,\,2}^*(t)-H_{K,\,2}^*(\lambda)+H_{K,\,1}^*(\lambda)\\[0.1cm]
	= {} & H_{K,\,2}^*(t)+h_{K+1}^*(\lambda,A_2,B_2+\pi/n)+h_{K+1}^*(\lambda,A_2,B_2)
\end{align*}
and put
\beq
  H_K^*(t) := \left\{
  \begin{array}{lcl}
	H_{K,\,1}^*(t) & \mbox{if} & t<\lambda\,,\\[0.2cm]
	\widetilde{H}_{K,\,2}^*(t) & \mbox{if} & t\geq\lambda
  \end{array}\right.
\eeq
so that $H_K^*$ is now a continuous function. This completes the proof. 
\end{proof}

\begin{corollary}
The distribution function $G$ of the distance $t$ between two random points in $\Po$ is given by
\beq
  G(t) = \left\{\begin{array}{lcl}
	0 & \mbox{if} & -\infty<s<0\,,\\[0.15cm]
	\dfrac{1}{A}\left[t^2\left(\pi-\dfrac{2u}{3A}\,t\right)+\dfrac{2u}{A}\,\phi^\n(t)\right] & 
		\mbox{if} & 0\leq t<\ell_{K+1}\,,\\[0.4cm]
	1 & \mbox{if} & t\geq \ell_{K+1}
\end{array}\right.
\eeq
with
\beq
  \phi^\n(t) = \sum_{\nu=0}^{k-1}K_\nu(\ell_{\nu+1})+K_k(t)
	\quad\mbox{if}\quad \ell_k\leq t<\ell_{k+1},\; k=0,\ldots,K\,,
\eeq
where
\beq
  K_k(t) = \frac{1}{2}\,\left(t^2-\ell_k^2\right)\left(\sum_{\nu=0}^{k-1}J_\nu^*(\ell_\nu,\ell_{\nu+1})-H_k^*(\ell_k)\right)
			+J_k^\n(\ell_k,t)
\eeq
with $J_k^*$ and $H_k^*$ according to Theorem \ref{Theorem2} and
\beq
  J_k^\n(s,t) = H_k^\n(t)-H_k^\n(s)\,,
\eeq
where
\beq
  H_k^\n(t) = \left\{
  \begin{array}{l}
	\dfrac{t^3}{3}-h_k^\n(t,A_1(k),B_1(k))+h_{k+1}^\n(t,A_1(k),B_1(k)+\pi/n)\\[0.25cm]
	\quad\mbox{if \, $(n$ is even $\,\wedge\,$ $k\in\{0,\ldots,K-1\})$ $\,\vee\,$
	$(n$ is odd $\,\wedge\,$ $t<\lambda)$}\,,\\[0.4cm]
	\dfrac{t^3}{3}-h_K^\n(t,A_1(K),B_1(K))+h_{K+1}^\n(t,A_1(K),B_1(K)+\pi/n)\\[0.3cm]
	-\:h_{K+1}^\n(t,A_2,B_2+\pi/n)-h_{K+1}^\n(t,A_2,B_2)\\[0.15cm]
	+\,\dfrac{t^2}{2}\left[h_{K+1}^*(\lambda,A_2,B_2+\pi/n)+h_{K+1}^*(\lambda,A_2,B_2)\right]\\[0.3cm]
	\quad\mbox{if \, $n$ is odd $\,\wedge\,$ $t\geq\lambda$}\,,\\[0.4cm]	
	\dfrac{t^3}{3}-h_K^\n(t,A_1(K),B_1(K))\mbox{ \; if \; $n$ is even $\,\wedge\,$ $k=K$}
  \end{array}\right.
\eeq
with $h_{K+1}^*$ from Theorem \ref{Theorem2}, and
\beq
  h_k^\n(t,a,b) = \left\{\begin{array}{l}
	0 \quad\mbox{if}\quad k=0 \,\vee\, (k\not=0\wedge t=0)\,,\\[0.1cm]
	\displaystyle{\sum_{i=1}^4\Theta_i(k,a,b)\,L_i^\n(t,a)} \quad\mbox{if}\quad k\not=0\,\wedge\,t>0\,,
  \end{array}\right.
\eeq
\vspace{-0.4cm}
\begin{align*}
  L_1^\n(t,a) = {} & \frac{t^4}{8} \;,\;\; L_2^\n(t,a) = \frac{t^2}{4}\,(2\ln t-1)\;,\\[0.1cm] 
  L_3^\n(t,a) = {} & \frac{1}{3}\left(t^2-a^2\right)^{3/2}+\frac{a}{2}\left(a\sqrt{t^2-a^2}
						+t^2\,\arcsin\frac{a}{t}\right),\\[0.1cm]
  L_4^\n(t,a) = {} & \frac{1}{8}\left[\frac{5a}{3}\left(t^2-a^2\right)^{3/2}
						+a^3\,\sqrt{t^2-a^2}+t^4\arcsin\frac{a}{t}\right].
\end{align*}
\end{corollary}

\begin{proof}
For $0\leq t<\ell_{K+1}$ one gets
\begin{align*}
  G(t) 
	= {} & \int_0^t g(\tau)\,\dd\tau 
	= \int_0^t\left(\frac{2\pi\tau}{A}-\frac{2u\tau^2}{A^2}+\frac{2u\tau}{A^2}\int_0^\tau F(s)\,\dd s\right)\dd\tau\\
	= {} & \frac{\pi t^2}{A}-\frac{2ut^3}{3A^2}+\frac{2u}{A^2}\int_0^t\tau\left(\int_0^\tau F(s)\,\dd s\right)\dd\tau\\
	= {} & \frac{\pi t^2}{A}-\frac{2ut^3}{3A^2}+\frac{2u}{A^2}\int_0^t\tau\phi^*(\tau)\,\dd\tau
	= \frac{1}{A}\left[t^2\left(\pi-\dfrac{2u}{3A}\,t\right)+\dfrac{2u}{A}\,\phi^\n(t)\right]
\end{align*}
with
\beq
  \phi^\n(t) := \int_0^t s\phi^*(s)\,\dd s\,.
\eeq 
It remains to calculate $\phi^\n(t)$. For $\ell_k\leq t<\ell_{k+1}$, $k=0,\ldots,K$, we have
\begin{align*}
  \phi^\n(t) 
	= {} & \int_0^{\ell_k}s\phi^*(s)\,\dd s + \int_{\ell_k}^t s\phi^*(s)\,\dd s\\
	= {} & \sum_{\nu=0}^{k-1}\,\int_{\ell_\nu}^{\ell_{\nu+1}}s\phi^*(s)\,\dd s
			+\int_{\ell_k}^t s\phi^*(s)\,\dd s
\end{align*}
with
\begin{align*}
  \int_{\ell_k}^t s\phi^*(s)\,\dd s   
	= {} & \int_{\ell_k}^t s\left(\sum_{\nu=0}^{k-1}J_\nu^*(\ell_\nu,\ell_{\nu+1})
			+J_k^*(\ell_k,s)\right)\dd s\\
	= {} & \sum_{\nu=0}^{k-1}J_\nu^*(\ell_\nu,\ell_{\nu+1})\int_{\ell_k}^t s\,\dd s
			+\int_{\ell_k}^t sJ_k^*(\ell_k,s)\,\dd s\\
	= {} & \sum_{\nu=0}^{k-1}J_\nu^*(\ell_\nu,\ell_{\nu+1})\int_{\ell_k}^t s\,\dd s
			+\int_{\ell_k}^t s\left[H_k^*(s)-H_k^*(\ell_k)\right]\dd s\\
	= {} & \left(\sum_{\nu=0}^{k-1}J_\nu^*(\ell_\nu,\ell_{\nu+1})-H_k^*(\ell_k)\right)\,\int_{\ell_k}^t s\,\dd s
			+\int_{\ell_k}^t s H_k^*(s)\,\dd s\,.
\end{align*}
Putting
\beq
  H_k^\n(t) := \int_{\ell_k}^t s H_k^*(s)\,\dd s \quad\mbox{and}\quad J_k^\n(s,t) := H_k^\n(t)-H_k^\n(s)\,,
\eeq
it follows that
\begin{align*}
  \int_{\ell_k}^t s\phi^*(s)\,\dd s   
	= {} & \frac{1}{2}\big(t^2-\ell_k^2\big)\left(\sum_{\nu=0}^{k-1}J_\nu^*(\ell_\nu,\ell_{\nu+1})-H_k^*(\ell_k)\right)
			+H_k^\n(t)-H_k^\n(\ell_k)\\
	= {} & \frac{1}{2}\big(t^2-\ell_k^2\big)\left(\sum_{\nu=0}^{k-1}J_\nu^*(\ell_\nu,\ell_{\nu+1})-H_k^*(\ell_k)\right)
			+J_k^\n(\ell_k,t)=:K_k(t)
\end{align*}
and hence
\beq
  \phi^\n(t) = \sum_{\nu=0}^{k-1}K_\nu(\ell_{\nu+1})+K_k(t) \,,\;\; \ell_k\leq t<\ell_{k+1} \,,\;\; k=0,\ldots,K.
\eeq
If $k\not=0$ and $t>0$, one finds
\begin{align*}
  h_k^\n(t,a,b)
	:= {} & \int th_k^*(t,a,b)\,\dd t = \sum_{i=1}^4\Theta_i(k,a,b)\int tL_i^*(t,a)\,\dd t\\[-0.2cm]
	 = {} & \sum_{i=1}^4\Theta_i(k,a,b)\,L_i^\n(t,a)
\end{align*}
with the indefinite integrals
\beq
  L_1^\n(t,a) = \int\frac{t^3}{2}\:\dd t = \frac{t^4}{8} \;,\quad L_2^\n(t,a)=\int t\ln t\:\dd t 
			  = \frac{t^2}{4}\,(2\ln t-1)
\eeq
and
\begin{align*}
  L_4^\n(t,a)
	= {} & \int tL_4^*(t,a)\,\dd t = \frac{1}{2}\int t\left(a\sqrt{t^2-a^2}+t^2\arcsin\frac{a}{t}\right)\dd t\\
	= {} & \frac{1}{2}\left(a\int t\,\sqrt{t^2-a^2}\;\dd t+\int t^3\arcsin\frac{a}{t}\;\dd t\right)
\end{align*}
with
\beq
  \int t\,\sqrt{t^2-a^2}\;\dd t = \frac{1}{3}\left(t^2-a^2\right)^{3/2} \quad\mbox{\cite[p.\ 47, Eq.\ 214]{Bronstein}}\,.
\eeq
Using integration by parts, we find
\beq
  \int t^3\arcsin\frac{a}{t}\;\dd t
	= \frac{1}{4}\left(t^4\arcsin\frac{a}{t}+a\int\frac{t^2}{\sqrt{1-(a/t)^2}}\;\dd t\right).
\eeq
Since $t\geq a>0$ in the present cases, 
\beq
  \int\frac{t^2}{\sqrt{1-(a/t)^2}}\;\dd t = \int\frac{t^3}{\sqrt{t^2-a^2}}\;\dd t\,, 
\eeq
and
\beq
  \int\frac{t^3}{\sqrt{t^2-a^2}}\;\dd t
	= \frac{1}{3}\left(t^2-a^2\right)^{3/2}+a^2\,\sqrt{t^2-a^2} \quad\mbox{\cite[p.\ 48, Eq.\ 223]{Bronstein}}\,.
\eeq
This yields
\beq
  L_4^\n(t,a) = \frac{1}{8}\left[\frac{5a}{3}\left(t^2-a^2\right)^{3/2}
				+a^3\,\sqrt{t^2-a^2}+t^4\arcsin\frac{a}{t}\right].
\eeq
Furthermore,
\begin{align*}
  L_3^\n(t,a) = {} & \int t\,\sqrt{t^2-a^2}\,\dd t + a\int t\arcsin\frac{a}{t}\,\dd t\\
			  = {} & \frac{1}{3}\left(t^2-a^2\right)^{3/2}+\frac{a}{2}\left(a\sqrt{t^2-a^2}
						+t^2\,\arcsin\frac{a}{t}\right)	
\end{align*}
(see the calculations of $L_4^\n(t,a)$ and $L_4^*(t,a)$).
\end{proof}

\section{Examples}

Fig.\ \ref{Fig_3} shows examples for chord length distribution functions $F$.

As special case of Theorem \ref{Theorem2}, the distance density function for an equilateral triangle $\mathcal{P}_{3,\,r}$ with circumscribed circle of radius $r$ is given by
\beq
  g(t) = \left\{\begin{array}{lll}
	\dfrac{2t}{A}\,\bigg[\pi+\dfrac{u}{A}\,\big(\phi(t)-t\big)\bigg] & \mbox{if} & t\in[0,\sqrt{3}\,r)\,,\\[0.4cm]
	0 & \mbox{if} & t\in\mathbb{R}\setminus [0,\sqrt{3}\,r)
  \end{array}\right.
\eeq
with $u=3\,\sqrt{3}\,r$, $A=\dfrac{3}{4}\,\sqrt{3}\,r^2$ and
\vspace{0.1cm}
\beq
  \phi(t) = \left\{\begin{array}{lll}
	\dfrac{\big(3\,\sqrt{3}+2\pi\big)t^2}{36 r} & \mbox{if} & 0\leq t<\dfrac{3r}{2}\,,\\[0.4cm]
	\dfrac{3}{2}\left[t\,\sqrt{1-\bigg(\dfrac{3r}{2t}\bigg)^2}-\dfrac{\pi r}{2}\right]
	+ \bigg(\dfrac{1}{4\,\sqrt{3}}-\dfrac{\pi}{9}\bigg)\dfrac{t^2}{r}\\[0.6cm]
	\qquad +\;\bigg(\dfrac{3 r}{2}+\dfrac{t^2}{3r}\bigg)\arcsin\dfrac{3r}{2t} & \mbox{if} & 
		\dfrac{3r}{2}\leq t<\sqrt{3}\,r\,.
  \end{array}\right.
\eeq
Fig.\ \ref{Fig_4} shows the function $r\times g(t)$ for $\mathcal{P}_{3,\,r}$ and some other examples.

One finds for the expectation of the distance for $\mathcal{P}_{3,\,r}$:
\begin{align*}
  \mathrm{E}[t]
	= {} & \int_0^{\sqrt{3}\,r}t\,g(t)\,\dd t 
	= \Bigg(\int_0^{3r/2}+\,\int_{3r/2}^{\sqrt{3}\,r}\Bigg)\,t\,g(t)\,\dd t\displaybreak[0]\\
	= {} & \frac{r}{20}\big(27-90\,\sqrt{3}+26\,\sqrt{3}\,\pi\big)
	       + \frac{r}{20}\big(-27+94\,\sqrt{3}-26\,\sqrt{3}\,\pi\\
	     & +\sqrt{3}\,\ln 27\big)
	= \frac{\sqrt{3}\,r}{20}\,\big(4+3\ln 3)\,.
\end{align*}
Since the side length $a$ of $\mathcal{P}_{3,\,r}$ is equal to $\sqrt{3}\,r$, we get
\beq
  \mathrm{E}[t] = \frac{a}{20}\,\big(4+3\ln 3) = \frac{3a}{5}\bigg(\frac{1}{3}+\frac{1}{4}\ln 3\bigg).
\eeq
This is the result from \cite[p.\ 49]{Santalo}.
\begin{figure}[h]
  \vspace{0cm}
  \begin{center}
    \includegraphics[scale=1.05]{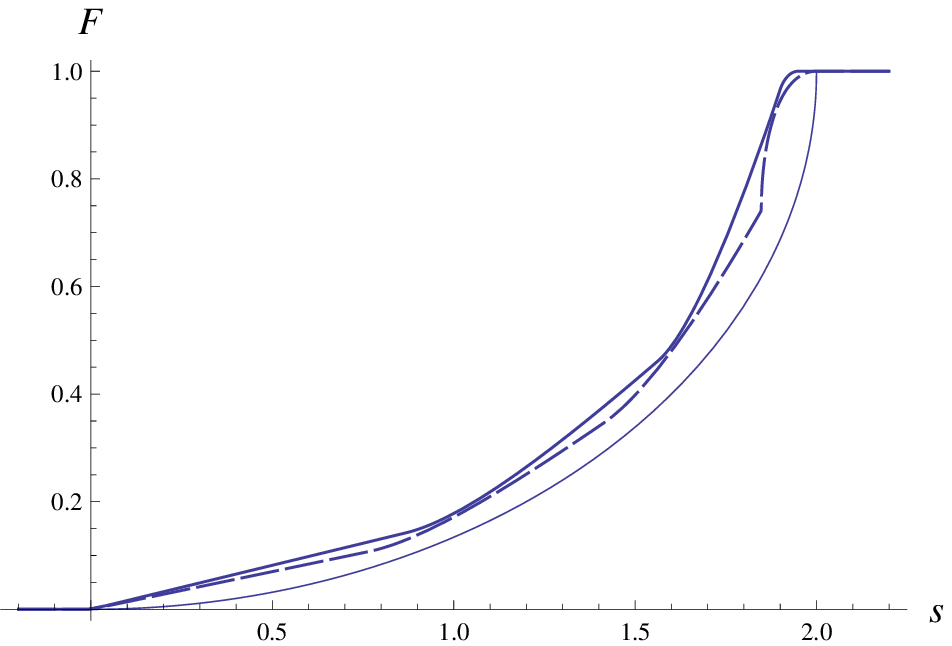}
  \end{center}
  \vspace{-0.67cm}
  \caption{\label{Fig_3} $F$ for $\mathcal{P}_{7,\,1}$ (thick), $\mathcal{P}_{8,\,1}$ (dashed) and circle of radius $r=1$ (thin)}
  \vspace{0.2cm}
  \begin{center}
    \includegraphics[scale=1]{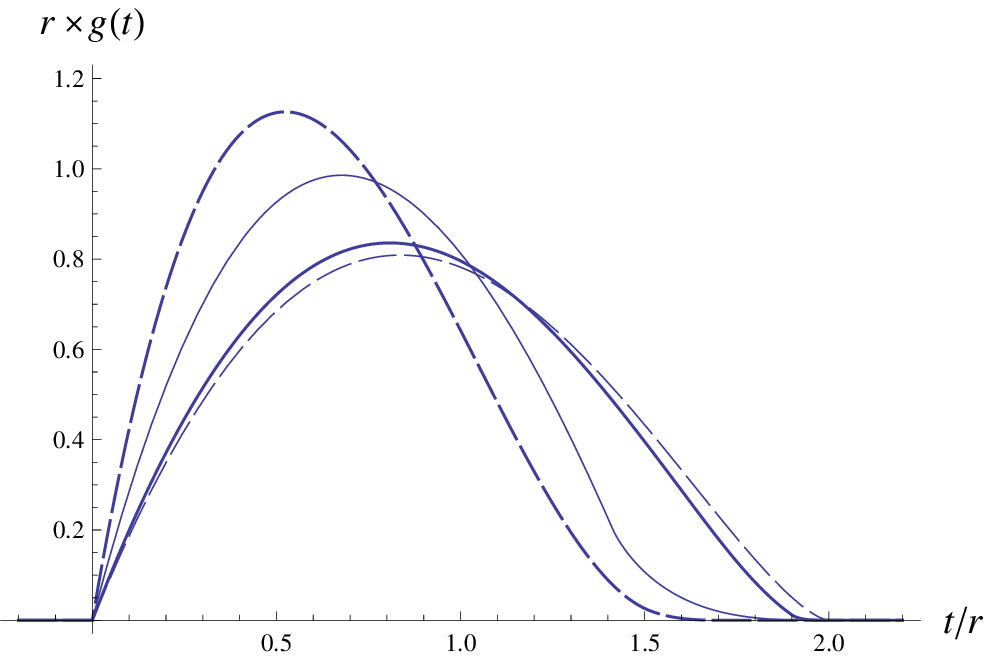}
  \end{center}
  \vspace{-0.67cm}
  \caption{\label{Fig_4} $r\times g(t)$ for $\mathcal{P}_{3,\,r}$ (thick and dashed), $\mathcal{P}_{4,\,r}$ (thin), $\mathcal{P}_{10,\,r}$ (thick) and circle of radius $r$ (thin and dashed)}
  \vspace{-0.5cm}
\end{figure}

\newpage
\vspace{0.2cm}
\noindent
{\large {\bf Acknowledgment}}\\[0.1cm]
I wish to thank Lothar Heinrich (University of Augsburg) for bringing Piefke's paper to my attention.

\vspace{0.4cm}
\begin{center} 
Uwe B\"asel\\[0.2cm] 
HTWK Leipzig, University of Applied Sciences,\\
Faculty of Mechanical and Energy Engineering,\\
PF 30 11 66, 04251 Leipzig, Germany
\end{center}
\end{document}